\documentclass[a4paper,10pt]{article}
\usepackage[utf8x]{inputenc}

\usepackage{amssymb}
\usepackage{amsmath,amsfonts,amsthm,amstext}
\usepackage[all,cmtip]{xy}
\usepackage[dvips]{graphicx}
\usepackage{graphics}
\usepackage{calrsfs}
\usepackage{hyperref}
\usepackage{xcolor}

\newtheorem{theorem}{Theorem}[section]
\newtheorem{lemma}[theorem]{Lemma}
\newtheorem{prop}[theorem]{Proposition}

\newtheorem{crlr}[theorem]{Corollary}

\theoremstyle{definition}
\newtheorem{defi}[theorem]{Definition}

\newcommand{\mb}[1]{\mathbb{#1}}

\newcommand{\ten}{\hat{\otimes}}
\newcommand{\ext}{\hat{\wedge}}

\def\X{\mathfrak{X}}
\begin{document}
\title{Weak commutativity for pro-$p$ groups}
 \author{Dessislava H. Kochloukova, Lu\'is Mendon\c{c}a\\
 	Department of Mathematics, \\ State University of Campinas, Brazil}

 \date{}
\maketitle
 

\begin{abstract}
 We define and study a pro-$p$ version of Sidki's weak commutativity construction. This is the pro-$p$ group $\X_p(G)$ generated by two copies $G$ and $G^{\psi}$ of a pro-$p$ group, subject to the defining relators $[g,g^{\psi}]$ for all $g \in G$. We show for instance that if $G$ is finitely presented or analytic pro-$p$, then $\X_p(G)$ has the same property.  Furthermore we study properties of the non-abelian tensor product and the pro-$p$ version of Rocco's construction $\nu(H)$. We also study finiteness properties of subdirect products of pro-$p$ groups. In particular we prove a pro-$p$ version of the $(n-1)-n-(n+1)$ Theorem. 
\end{abstract}

\begin{section}{Introduction} 

In \cite{Sid} Sidki defined for an arbitrary discrete group $H$ the group $\X(H)$ and initiated the study of this  construction. The case when $H$ is nilpotent was studied  by Gupta, Rocco and Sidki in \cite{GupRocSid}, \cite{Sid}. There are links between the weak commutativity construction and homology, in particular Rocco showed in \cite{Norai2} that the Schur multiplier of $H$ is a subquotient of $\X(H)$ isomorphic to $W(H)/R(H)$, where $W(H)$ and $R(H)$ are special normal subgroups of $\X(H)$ defined in \cite{Sid}. Recently Lima and Oliveira used in \cite{L-O} the Schur multiplier of $H$  to show that for any virtually polycyclic group $H$ the group  $\X(H)$  is virtually polycyclic too.  The use of homological methods in the study of $\X(H)$ was further developed by Bridson, Kochloukova and Sidki in \cite{BriKoc}, \cite{KocSid}, where finite presentability and the homological finiteness type $FP_m$ of $\X(H)$ were studied. Recently  Kochloukova and Mendon\c{c}a calculated in \cite{K-M} low dimensional cases of the Bieri-Strebel-Neumann-Renz $\Sigma$-invariants of $\X(H)$.

In this paper we study a pro-$p$ version of the construction $\X(H)$ for a fixed prime $p$. 
	Let $G$ be a pro-$p$ group. We define $\X_p(G)$ by the pro-$p$ presentation
	\[ \X_p(G) = \langle G, G^{\psi} \hbox{ } | \hbox{ } [g,g^{\psi}]=1 \hbox{ for all } g \in G\rangle_p,\]
	where $G^{\psi}$ is an isomorphic copy of $G$ via $g \mapsto g^{\psi}$ and $\langle - \mid  - \rangle_p$ denotes presentation by generators and relators in the category of pro-$p$ groups.
	
	We develop structure theory of $\X_p(G)$ similar to the discrete case by defining special normal subgroups $D_p(G)$, $L_p(G)$, $W_p(G)$ and $R_p(G)$. The structure theory we develop shows that $\X_p(G)/ W_p(G)$ is a subdirect product of $G \times G \times G$. We need a criterion for homological finiteness properties $FP_m$ of pro-$p$ subdirect products of pro-$p$ groups. The case of a pro-$p$ subdirect product  inside a direct product of free pro-$p$ or Demushkin group was considered by Kochloukova and Short in \cite{KocSho}. The    homological finiteness properties $FP_m$ of subdirect products
	in the abstract case were studied by Bridson, Howie, Kuckuck, Miller and  Short in \cite{B-H-M-S}, \cite{Kuc} and were reduced to the Fibre Product Conjecture that treats the homological finiteness type of a fibre product of groups and is known to hold  in small dimensions \cite{B-H-M-S}. We show in Section \ref{pro-p-fibre} that a version of  the Fibre Product Conjecture for pro-$p$ groups holds in any dimension and it is a corollary of the methods  already developed for abstract groups by Kuckuck in \cite{Kuc}.
	
	\medskip{\bf The $(n-1)-n-(n+1)$ Theorem for pro-$p$ groups} {\it Let $p_1 : G_1 \to Q$ and $p_2 : G_2 \to Q$ be surjective homomorphisms of pro-$p$ groups. Suppose that $Ker (p_1)$ is of type $FP_{n-1}$, both $G_1$ and $G_2$ are of type $FP_n$ and $Q$ is of type $FP_{n+1}$. Then the fiber pro-$p$ product
	$$
	G = \{ (g_1, g_2) \mid g_1 \in G_1, g_2 \in G_2, p_1(g_1) = p_2(g_2) \}
	$$
	has type $FP_n$.}

\medskip As a corollary we deduce the following result.

\medskip
{\bf Corollary A} {\it 	Let $G_1, \ldots, G_n$ be pro-$p$ groups of type $FP_k$ for some $n \geq 1$. Denote by $p_{i_1, \ldots, i_k}$ the projection $G_1 \times \dots \times G_n \twoheadrightarrow G_{i_1} \times \dots \times G_{i_k}$ for $1 \leq i_1 <  \ldots <  i_k\ \leq n $. Let $H \subseteq G_1 \times \dots \times G_n$ be a closed subgroup such that $p_{i_1, \ldots, i_k}(H)$ is of finite index in $G_{i_1} \times \dots \times G_{i_k}$ for all   $1 \leq i_1 <  \ldots <  i_k\ \leq n $.
 Then $H$ is of type $FP_k$.  }

\medskip
We show that the construction $\X_p(G)$ has similar properties to the original weak commutativity construction $\X(H)$ for a discrete group $H$. 

\medskip
{\bf Theorem B} {\it Let $G$ be a pro-$p$ group. Then

1. If $G$ is $p$-finite then $\X(G) \simeq \X_p(G)$;

2. If $\mathcal{P}$ is one of the following classes of pro-$p$ groups : soluble; finitely generated nilpotent; finitely presented; poly-procyclic; analytic pro-$p$  and $G \in \mathcal{P}$ then $\X_p(G) \in \mathcal{P}$;

3. If $G$ is a free non-procyclic finitely generated pro-$p$ group then $\X_p(G)$ is not of homological type $FP_3$.}

\medskip
In order to prove Theorem B  we use the notion of non-abelian tensor product of pro-$p$ groups developed by Moravec in \cite{Mor} and we show that $H_2(G, \mathbb{Z}_p)$ is a subquotient of $\X_p(G)$. More precisely it is isomorphic to $W_p(G)/ R_p(G)$.  In the case of discrete groups the non-abelian tensor product was introduced by Brown and Loday in \cite{B-L}, following ideas of Dennis in \cite{D}. The non-abelian tensor product $H \otimes H$ for discrete groups $H$ is linked with the Rocco construction $\nu(H) = \langle H, H^{\psi} \mid [h_1, h_2^{\psi}]^{h_3} = [ h_1^{h_3}, (h_2^{h_3})^{\psi}] = [h_1, h_2^{\psi}]^{h_3^{\psi}}\rangle$ defined in  \cite{Norai}. In section \ref{sectionRocco} we define a pro-$p$ version $\nu_p(G)$ for a pro-$p$ group $G$ and study its properties.

In part 2 of Theorem B we consider several classes of pro-$p$ groups, including the class of analytic pro-$p$ groups introduced  first by Lazard in \cite{L}. A pure group theoretic approach to analytic pro-$p$ groups was developed by  Lubotzky, Mann, du Sautoy and Segal in \cite{bookDSMS}. We show that the constructions $\X_p$ and $\nu_p$ preserve analytic pro-$p$ groups. Furthermore we show that part 2 of Theorem B has a version for non-abelian tensor product of pro-p groups and the $\nu_p$-construction.

\medskip
{\bf Theorem C} {\it If $\mathcal{P}$ is one of the following classes of pro-$p$ groups: soluble; finitely generated nilpotent; finitely presented; poly-procyclic; analytic pro-$p$  and $G \in \mathcal{P}$ then $\nu_p(G) \in \mathcal{P}$ and, except in the case when ${\mathcal P}$  is the class of finitely presented groups, the non-abelian tensor pro-$p$ product $G \widehat{\otimes} G \in \mathcal{P}$.}

\medskip
{\bf Acknowledgements} During the preparation of this work the first named author was partially supported by CNPq grant 301779/2017-1  and by FAPESP grant 2018/23690-6. The second named author was supported by PhD grant FAPESP 2015/22064-6.
\end{section}

\begin{section}{Preliminaries on homological properties of pro-$p$ groups}

Throughout this paper $p$ denotes a fixed  prime number.
Let $G$ be  a pro-$p$ group. Thus $$G = \varprojlim G/U,$$
where $U$ runs through all finite quotients $G/U$ of $G$ and note that these finite quotients are actually finite $p$-groups. The completed group algebra is defined as
$$ \mathbb{Z}_p[[G]] = \varprojlim \mathbb{Z}_p[G/U],$$
where $U$ runs through all finite quotients $G/U$ of $G$.
There is a homology and cohomology theory for pro-$p$ groups, more generally for profinite
groups, developed in \cite{RibZal}. In particular there are free resolutions of pro-$p$ $\mathbb{Z}_p[[G]]$-modules $\mathcal P$ that can be used to calculate continuous homology groups. By definition for a left pro-$p$ $\mathbb{Z}_p[[G]]$-module $A$ the continuous homology group $H_i(G, A)$ is the homology of the complex ${\mathcal P} \widehat{\otimes}_{\mathbb{Z}_p[[G]]} A$, where $\widehat{\otimes}$ denotes the completed tensor product.

 A pro-$p$ group $G$ has homological type $FP_m$ if the trivial $\mathbb{Z}_p[[G]]$-module $\mathbb{Z}_p$ has a continuous free resolution with all free modules finitely generated in dimension $\leq m$. This turns out to be equivalent in the case of a pro-$p$ group $G$ to all continuous homology groups $H_i(G, \mathbb{F}_p)$  being finite for $i \leq m$ \cite{Kin}. This is equivalent to all continuous homology groups $H_i(G, \mathbb{Z}_p)$ being finitely generated pro-$p$ groups for $i \leq m$ \cite{Kin}. Note that for discrete groups this equivalence is not valid; the property  $FP_m$ implies that the homology groups are finitely generated as discrete groups but the converse does not hold.
 
 Note that for pro-$p$ groups $G$ finite presentability is equivalent to the finiteness of both $H^1(G, \mathbb{F}_p)$
and $H^2(G, \mathbb{F}_p)$. And this is equivalent to the finiteness of both $H_1(G, \mathbb{F}_p)$
and $H_2(G, \mathbb{F}_p)$. Thus $G$ is a finitely presented pro-$p$ group if and only if $G$ is $FP_2$. Note that this is a very specific property of pro-$p$ groups that does not hold for discrete groups as Bestvina and Brady have shown in \cite{B-B} examples of discrete groups of type $FP_2$ that are not finitely presented.

\section{Pro-$p$ fibre and subdirect products} \label{pro-p-fibre}

In \cite{Kuc} Kuckuck discussed in great details when the fiber product of abstract groups has homotopical type $F_m$. The general case is still open, but in \cite{Kuc} the new notion of weak $FP_m$ type was introduced and more results for this weaker homological type were proved. 
By definition a discrete group $H$ is $wFP_m$ ( i.e. weak $FP_m$) if for every subgroup $H_0$ of finite index in $H$ we have that $H_i(H_0, \mathbb{Z})$ is finitely generated for all $i \leq m$.

Here we show that in the case of pro-$p$ groups it is relatively easy to obtain a general result for the homological type of a fiber product. The explanation for this is that for a pro-$p$ group $G$  we have that $G$ is of type $FP_m$ if and only if $H_i(G, \mathbb{Z}_p)$ is finitely generated as a pro-$p$ group  for all $ i \leq m$. Thus $G$ is of type $FP_m$ if and only if $H_i(G_0, \mathbb{Z}_p)$ is finitely generated as a pro-$p$ group  for all $ i \leq m$ and all pro-$p$ subgroups $G_0$ of finite index in $G$ . The last condition is a pro-p version of the  definition of homological type weak $FP_m$ (denoted $wFP_m$ by Kuckuck in \cite{Kuc}), thus for pro-$p$ groups the notions of type $FP_m$ and weak $FP_m$ are the same.

	\medskip{\bf The $(n-1)-n-(n+1)$ Theorem for pro-$p$ groups} {\it Let $p_1 : G_1 \to Q$ and $p_2 : G_2 \to Q$ be surjective homomorphisms of pro-$p$ groups. Suppose that $Ker (p_1)$ is of type $FP_{n-1}$, both $G_1$ and $G_2$ are of type $FP_n$ and $Q$ is of type $FP_{n+1}$. Then the fiber pro-$p$ product
	$$
	G = \{ (g_1, g_2) \mid g_1 \in G_1, g_2 \in G_2, p_1(g_1) = p_2(g_2) \}
	$$
	has type $FP_n$.}

\medskip \begin{proof}  The proof of the following claim is identical with the proof of \cite[Lemma~5.3]{Kuc}
	subject to changing the category of discrete groups with the category of pro-$p$ groups, discrete homology with continuous homology, weak $FP_m$ for discrete groups with $FP_m$ for pro-$p$ groups and coefficient module $\mathbb{Z}$ in the discrete case with coefficient module $\mathbb{F}_p$ in the pro-$p$ case. The proof of  \cite[Lemma~5.3]{Kuc}
	is homological and uses spectral sequences.
	
	\medskip
	{\bf Claim} {\it Let $N \to \Gamma \to Q$ be a short exact sequence of pro-$p$ groups, where $N$ is $FP_{n-1}$, and $Q$ is $FP_{n+1}$. Then $H_0(Q, H_n(N, \mathbb{F}_p))$ is finite if and only if  $\Gamma$ is $FP_n$.}
	
	\medskip
	By the above Claim applied for the short exact sequence $Ker (p_1) \to G_1 \to Q$ we deduce that 
	\begin{equation} \label{border}
	H_0(Q, H_n(Ker(p_1), \mathbb{F}_p)) \hbox{ is finite.}
	\end{equation}
	Then we follow the proof of \cite[Thm.~5.4]{Kuc} changing the coefficient module in the homology groups to $\mathbb{F}_p$ and consider the short exact sequence of groups
	\begin{equation}
	Ker (p_1) \to G \to G_2
	\end{equation} 
	and the corresponding LHS spectral sequences
	$$
	E_{i,j}^2 = H_i(G_2, H_j(Ker (p_1), \mathbb{F}_p)) \hbox{ converging to } H_{i+j}(G, \mathbb{F}_p).
	$$
	By construction \begin{equation} E_{i,j}^2 \hbox{ is finite for }i \leq n, j \leq n-1. \end{equation}  Indeed, since $Ker(p_1)$ is $FP_{n-1}$, $H_j(Ker (p_1), \mathbb{F}_p)$ is a finite $\mathbb{Z}_p[[G_2]]$-module for $j \leq n-1$ and has a filtration, where each quotient $W_s$ is a simple $\mathbb{Z}_p[[G_2]]$-module.
	Any simple  $\mathbb{Z}_p[[G_2]]$-module is isomorphic to the trivial $\mathbb{Z}_p[[G_2]]$-module $\mathbb{F}_p$ and since $G_2$ is $FP_n$ we see that $H_i(G_2, \mathbb{F}_p)$ is finite for $i \leq n$. Hence $H_i(G_2, W_s)$ is finite for every $s$ and so $	E_{i,j}^2$ is finite for $i \leq n, j \leq n-1$.

	Note that by (\ref{border})
	$$
	E_{0,n}^2 = H_0(G_2, H_n(Ker (p_1), \mathbb{F}_p)) = H_0(Q, H_n(Ker (p_1), \mathbb{F}_p)) \hbox{ is finite.}
	$$
	Hence by the convergence of the spectral sequence $H_k(G, \mathbb{F}_p)$ is finite for $k \leq n$, so $G$ has type $FP_n$.
	\end{proof}

In a pro-$p$ group $G$ for a subset $S$ of $G$ we denote by $\overline{\langle S \rangle}$ the  pro-$p$ subgroup of $G$ generated by $S$ and
we denote by $G'$ the commutator subgroup $\overline{[G,G]}$.

\begin{crlr} \label{homology} Let $H$ be a finitely presented pro-$p$ group   and $S = \overline{\langle (h, h^{-1}) \mid h \in H \rangle} \leq H \times H$. Suppose that $H'$ is finitely generated as a pro-$p$ group. Then $S$ is a finitely presented pro-$p$ group.
\end{crlr} 

\begin{proof} Note that $H'\times H' \subseteq S$. Let $p_1 : H \to Q = H/ H'$ be the canonical projection and $p_2 = \sigma \circ p_1$, where $\sigma : Q \to Q$ is the antipodal homomorphism sending $g$ to $g^{-1}$.Then $S$ is the fiber product of the maps $p_1$ and $p_2$ and we can apply the $(n-1)-n-(n+1)$ Theorem for pro-$p$ groups for $n = 2$ to deduce that $S$ has type $FP_2$, that for a pro-$p$ group is the same as being finitely presented.
\end{proof}

{\bf Proof of Corollary A}
	This is a pro-$p$ version of Corollary 5.5 in \cite{Kuc}. As in Corollary 5.5 in \cite{Kuc} there are suitable ways of decomposing $H$ as a fibre product  and then we can apply  the $(n-1)-n-(n+1)$ Theorem for pro-$p$ groups. 

\end{section}

\begin{section}{Finite presentability and completions}

\begin{defi}
Let $G$ be a pro-$p$ group. We define $\X_p(G)$ by the pro-$p$ presentation
\[ \X_p(G) = \langle G, G^{\psi} \hbox{ } | \hbox{ } [g,g^{\psi}] \hbox{ for  } g \in G\rangle_p,\]
where $G^{\psi}$ is an isomorphic copy of $G$ via $g \mapsto g^{\psi}$  and $\langle - | - \rangle_p$ denotes presentation by  generators and relators in the category of pro-$p$ groups. 
\end{defi}
 
We denote by $\X(H)$ the original Sidki's weak commutativity construction \cite{Sid} for a discrete group $H$, i.e. 
\[ \X(H) = \langle H, H^{\psi} \hbox{ } | \hbox{ } [h,h^{\psi}] \hbox{ for  } h \in H\rangle, 
\] where $H^{\psi}$ is an isomorphic copy of $H$ via $h \mapsto h^{\psi}$ and $\langle - | - \rangle$ denotes presentation by generators and relators in the category of discrete groups. 
 
\begin{lemma} \label{finite-p}
For a finite $p$-group $P$ we have $\X(P) \simeq \X_p(P)$.
\end{lemma}

\begin{proof}
 The construction $\X( G )$ for a discrete group $G$ is characterized by the following property: for any discrete group $\Gamma$ and two homomorphisms 
 $\sigma,  \tau: G \to \Gamma$ such that $[\sigma(g), \tau(g)] = 1$ for all $g \in G$, there exists a unique homomorphism $\varphi: \X(G) \to \Gamma$ such that $\varphi |_G = \sigma$ and $\psi \circ (\varphi |_{G^{\psi}}) = \tau$. The pro-$p$ construction has the analogous property where ``discrete'' is substituted with ``pro-$p$'', and ``homomorphism'' with ``continuous homomorphism''. This together with the fact that $\X(P)$ is a finite $p$-group allows us to construct the obvious maps between $\X(P)$ and $\X_p(P)$, in both directions, whose compositions are the identity maps and whose restriction on $P \cup P^{\psi}$ is the identity map.
\end{proof}

\begin{lemma} \label{inverse0} 
	Let $G$ be a pro-$p$ group.
	Then $\X_p(G) \simeq \varprojlim \X(G/U)$, where the inverse limit is over all finite quotients $G/U$ of $G$.
\end{lemma}

\begin{proof} Let $U$ be an open normal subgroup of $G$. Then the epimorphism $G \to G/ U$ induces an epimorphism  $\X_p(G) \to \X_p(G/U)$ and we identify $\X_p(G/U)$ with  $ \X(G/U)$. This induces an epimorphism
	$$\X_p(G) \to \varprojlim \X(G/U).
	$$
	To show that this map is an isomorphism it suffices to show that every finite $p$-group $V$ that is a quotient of $\X_p(G)$, is a quotient of suitable $\X(G/U)$. Let $\mu : \X_p(G) \to V$ be the quotient map and $U_1 = Ker (\mu) \cap G$ and $U_2 = Ker (\mu)^{\psi} \cap G$, i.e. $U_2^{\psi} = Ker(\mu) \cap G^{\psi}$. Set $U = U_1 \cap U_2$ and note that since $V$ is a finite $p$-group, the groups $G/ U_1$, $G/ U_2$ and $G/ U$  are finite $p$-groups. Since the original Sidki construction preserves finite $p$-groups we conclude that  $\X_p(G/U) \simeq \X(G/ U)$ is a finite $p$-group and by construction $V$ is a quotient of $\X_p(G/U)$ since $Ker (\X_p(G) \to \X_p(G/U))$ is the normal closed subgroup generated by $U \cup U^{\psi}$, hence $Ker (\X_p(G) \to \X_p(G/U)) \subseteq Ker (\mu)$.	
\end{proof}
For a discrete group $H$ we denote by $\widehat{H}$ the pro-$p$ completion of $H$.
 
\begin{prop} \label{p1}
 For any discrete group H, we have
 \[ \X_p(\widehat{H}) \simeq \widehat{\X(H)}.\]
\end{prop}

\begin{proof}
Let $H = \langle X \mid R \rangle$ be a presentation of $H$ as a discrete group. Note we are not assuming that $X$ or $R$ is finite.
 Notice that $\X(H)$ is the discrete group generated by $X \cup X^{\psi}$, with $\{[h,h^{\psi}]| h \in H \} \cup R \cup R^{\psi}$ as a set defining relators. Clearly $\widehat{\X(H)}$ is the group with this same presentation in the category of pro-$p$ groups, that is:
 \[ \widehat{\X(H)} = \langle X, X^{\psi} |R, R^{\psi},  [h,h^{\psi}] \hbox{ for } h \in i(H) \subset \widehat{H}\rangle_p,\]
 where $i: H \to \widehat{H}$ is the canonical map. The group $\X_p(\widehat{H})$, on the other hand, has by definition a presentation with the same set of generators, but with defining relators $[h,h^{\psi}]$ for all $h \in \widehat{H}$ (rather that only $h \in i(H)$). Thus there is a continuous epimorphism $\widehat{\X(H)} \twoheadrightarrow \X_p(\widehat{H})$.

 To show that this is an isomorphism it suffices to show that
every finite $p$-group quotient $V$ of $\widehat{\X(H)}$ is a finite $p$-quotient of $\X_p(\widehat{H})$. Note that $V$ is a finite $p$-group quotient of $\X(H)$ since $\widehat{\X(H)}$ is the pro-$p$ completion of $\X(H)$. Thus  if $\mu : \X(H) \to V$ is an epimorphism we set $U_1 = Ker (\mu) \cap H, U_2 = Ker (\mu)^{\psi} \cap H$ and $U = U_1 \cap U_2$. Then $V$ is a quotient of the finite $p$-group $\X(P)$, where $P = H/ U$ is a finite $p$-group and by Lemma \ref{inverse0} $\X(P)$ is a quotient of $\X_p( \widehat{H})$.
 \end{proof}

\begin{crlr} \label{fin-pres}
 If $G$ is a finitely presented pro-$p$ group, then so is $\X_p(G)$.
\end{crlr}

\begin{proof} 
	Let $F$ be a finitely generated free discrete group and let $\widehat{F}$ be its pro-$p$ completion i.e. it is a free pro-$p$ group with the same free basis as $F$. By Proposition \ref{p1} $\X_p(\widehat{F})$ is the pro-$p$ completion of $\X(F)$, which is finitely presented as a discrete group by Theorem A in \cite{BriKoc}. Thus $\X_p(\widehat{F})$ also admits a finite presentation (as a pro-$p$ group).
 
 In general, if $G \simeq \widehat{F}/R$, where $R \leq \widehat{F}$ is the normal closure of the closed subgroup generated by a finite set $\{r_1, \ldots, r_n\} \subset R$, then $\X_p(G)$ is the quotient of $\X_p(\widehat{F})$ by the normal closure of the closed subgroup generated by
 $\{r_1, \ldots, r_n, r_1^{\psi}, \ldots, r_n^{\psi}\}$. Thus $\X_p(G)$ is also finitely presented as a pro-$p$ group.
\end{proof}

Recall that for a finitely generated, nilpotent discrete group $H$, it was shown by
Gupta, Rocco and Sidki in  \cite{GupRocSid} that $\X(H)$ is nilpotent. The proof in \cite{GupRocSid} involves long commutator calculations. Here we prove a pro-$p$ version of this result as a corollary of the fact that it holds for discrete groups.
\begin{prop} \label{nilpotent1}
Let $G$ be a finitely generated nilpotent pro-$p$ group. Then $\X_p(G)$ is a nilpotent
pro-$p$ group.
\end{prop}

\begin{proof} Let $X$ be a finite generating set of $G$ as a pro-$p$ group
i.e. $G = \overline{\langle X \rangle}$. Let $H$ be the discrete subgroup of $G$ 
generated by $X$ i.e. $H = \langle X \rangle$. Then the closure
$\overline{H}$ of $H$ is $G$, hence there is an epimorphism of pro-$p$
groups $\widehat{H} \to \overline{H} = G$ that induces an epimorphism of
pro-$p$ groups $\X_p(\widehat{H}) \to \X_p(G)$. Since $G$ is nilpotent,
$H$ is nilpotent and hence $\X(H)$ is nilpotent. Then $\X_p(\widehat{H}) \simeq 
\widehat{\X(H)}$ and its quotient $\X_p(G)$ are nilpotent pro-$p$ groups.
\end{proof}
\end{section}

\begin{section}{Some structural results}

 Recall that in the discrete case $L(H) \leq \X(H)$ denotes the subgroup generated by the elements $h^{-1} h^{\psi}$ for all $h \in H$. It can  be described also as the kernel of the homomorphism $\alpha: \X(H) \to H$ defined by $\alpha(h) = \alpha(h^{\psi}) = h$ for all $h \in H$.

 \begin{defi}
 {\it For a pro-$p$ group $G$, we define $L_p = L_p(G):= Ker(\alpha)$, where $\alpha: \X_p(G) \to G$ is the homomorphism of pro-$p$ groups defined by $\alpha(g) = \alpha(g^{\psi}) = g$ for all $g \in G$. }
 \end{defi}

 \begin{lemma}  \label{genL} Let $G$ be a pro-$p$ group. Then
  $L_p(G)$ is the closed subgroup generated by the elements $g^{-1} g^{\psi}$, for all $g \in G$.
 \end{lemma}

 \begin{proof}
  Denote by $A(G)$ the closed subgroup of $\X_p(G)$ generated by the elements $g^{-1} g^{\psi}$ for  $g \in G$. Note that $A(Q)= L_p(Q)$ for a
  finite $p$-group $Q$. Clearly $A(G) \subseteq Ker (\alpha)$. Let $\pi_0: G \twoheadrightarrow Q$ be a continuous epimorphism onto a finite $p$-group $Q$. Then we have a commutative diagram:
    \[      \xymatrix{ \X_p(G) \ar@{->>}[r]^{\pi} \ar@{->>}[d]_{\alpha} & \X(Q) \ar@{->>}[d]^{\alpha_Q} \\
                G \ar@{->>}[r]_{\pi_0} & Q}
  \] where $\pi$ is induced by $\pi_0$.
  If follows that $\pi(L_p(G)) =  Ker(\alpha_Q) = L_p(Q) = A(Q) = \pi(A(G))$. This holds for all $Q$, thus $L_p(G)$ and $A(G)$ cannot be distinguished by looking at the finite quotients of $\X_p(G)$ since by Lemma \ref{inverse0} $\X_p(G)$ is the inverse limit of $\X(Q)$ over all possible finite $p$-groups $Q$. So $L_p(G) = A(G)$.  
 \end{proof}
 
 \begin{prop}  \label{Lfg}
  If $G$ is a finitely generated pro-$p$ group, then $L_p(G)$ is a finitely generated pro-$p$ group.
 \end{prop}

 \begin{proof}
  Let $F$ be a finitely generated free discrete group and let $\widehat{F}$ be its pro-$p$ completion i.e. it is a free pro-$p$ group with the same free basis as $F$. We already know that $\X_p(\widehat{F}) \simeq \widehat{\X(F)}$. By \cite[Prop.~2.3]{BriKoc} we know that $L(F) \subseteq \X(F)$ is a finitely generated discrete group. If $Y \subseteq L(F)$ is a finite generating set, then its image $i(Y)$ in $\X_p(\widehat{F})$, where $i: \X(F) \to \widehat{\X(F)} \simeq \X_p(\widehat{F})$ is the canonical map, is a generating set for $L_p(\widehat{F})$ as a pro-$p$ group. Indeed, it suffices to verify that the image of $i(Y)$ in each quotient of $\widehat{\X(F)}$ of the type $\X(P)$, where $P$ is a finite $p$-group, is a generating set for $L(P)$. This follows from the fact that $L(P)$ is the image of $L(F)$ by the homomorphism $\X(F) \twoheadrightarrow \X(P)$ induced by the projection $F \twoheadrightarrow P$.
  
  Thus $L_p(\widehat{F})$ is a finitely generated pro-$p$ group. In general, if $G$ is a quotient of $\widehat{F}$, then $L_p(G)$ is a quotient of $L_p(\widehat{F})$, thus it is finitely generated as a pro-$p$ group too.
 \end{proof}

 \begin{defi}
{\it For a pro-$p$ group $G$, we define $$D_p = D_p(G):= Ker(\beta),$$ where $$\beta: \X_p(G) \to G \times G$$ is the homomorphism of pro-$p$ groups defined by 
$\beta(g) = (g,1)$ and $\beta(g^{\psi}) = (1,g)$ for all $g \in G$. }
 \end{defi}

 Arguing as in Lemma \ref{genL}, we obtain:
 
\begin{lemma} Let $G$ be a pro-$p$ group. Then
$D_p(G)$ is generated by $[g,h^{\psi}]$, for $g,h\in G$, as normal closed subgroup of $\X_p(G)$.
\end{lemma}

\begin{prop} \label{DLcommute} Let $G$ be a pro-$p$ group. Then
 $[L_p(G), D_p(G)]=1$.
\end{prop}

\begin{proof}
As in \cite[Lemma~4.1.6]{Sid}, for any $g_1,g_2,g_3 \in G$, the relation 
\[ [g_1^{-1} g_1^{\psi}, [g_2, g_3^{\psi}]]=1\]
can be obtained as a consequence of the defining relations $[h,h^{\psi}]=1$ for $h \in \{g_1g_2, g_2g_3, g_1g_3, g_1g_2g_3\}$. Thus the (topological) generators of $L_p$ commute with the (topological) generators of $D_p$ as a normal
subgroup of $\X_p(G)$, that is,  $[L_p,D_p]=1$.
\end{proof}

\begin{defi}
 {\it Define the homomorphism of pro-$p$ groups
 \[\rho_p: \X_p(G) \to G \times G \times G\]
 by $\rho_p(g) = (g,g,1)$ and $\rho_p(g^{\psi}) = (1,g,g)$.}
\end{defi}

Consider the diagonal map $d: \X_p(G) \to (\X_p(G) / L_p(G)) \times (\X_p(G) / D_p(G))$ that sends $g$ to $(g L_p(G), g D_p(G) )$ and the isomorphism $\nu : (\X_p(G) / L_p(G)) \times (\X_p(G) / D_p(G)) \to G \times G \times G $ induced by the isomorphisms $(\X_p(G) / L_p(G)) \times (\X_p(G) / D_p(G)) \simeq G \times G \times G^{\psi}  \simeq G \times G \times G$. Then $\sigma \circ  \nu  \circ d = \rho_p$, where $\sigma : G \times G \times G \to G \times G \times G$ is the isomorphism that sends $(g_1, g_2, g_3)$ to $(g_2, g_1, g_3)$. Then $Ker (\rho_p) = Ker (d) = D_p(G) \cap L_p(G)$.
We set 
\[W_p = W_p(G):=Ker(\rho_p) = D_p(G) \cap L_p(G).\]
This is a normal subgroup of $\X_p(G)$ which is
central in $D_p(G) L_p(G)$ by Proposition \ref{DLcommute}. In particular, $W_p(G)$ is abelian. Furthermore, it is not hard to see that the image of $\rho_p$ can be written as
\begin{equation} \label{Imrho}
 Im(\rho_p) = \{ (g_1,g_2,g_3) \in G \times G \times G \hbox{ } | \hbox{ } g_1 g_2^{-1} g_3 \in \overline{[G,G]} \}.
\end{equation}

\begin{prop}  \label{Solv}
 If $G$ is soluble, then so is $\X_p(G)$.
\end{prop}
\begin{proof}
 This is clear, since $W_p$ is abelian and $\X_p(G)/W_p$ is a subgroup of $G \times G \times G$, therefore is soluble.
\end{proof}

 It is clear by \eqref{Imrho} that $Im(\rho_p)$ is a subdirect product of $G \times G \times G$, i.e. maps surjectively on each copy of $G$ in the direct product.

 \begin{crlr}
  If $G$ is a finitely presented pro-$p$ group, then $\X_p(G)/W_p(G) \simeq Im(\rho_p)$ is a finitely presented pro-$p$ group. 
 \end{crlr}
 
 \begin{proof} Recall that a pro-$p$ group is finitely presented if and only if it is $FP_2$.
  It is clear by \eqref{Imrho} that 
  \[p_{1,2}(Im(\rho_p)) =  p_{1,3}(Im(\rho_p)) = p_{2,3}(Im(\rho_p)) =  G \times G.\]
  Thus Corollary A applies for $k = 2$.
 \end{proof}

 \begin{prop}  \label{Wfg}
  If $G$ is a finitely presented pro-$p$ group  such that $\overline{[G,G]}$ is a finitely generated pro-$p$ group, then $W_p(G)$ is a finitely generated pro-$p$ group.
 \end{prop}     

 \begin{proof}
  The proof is similar to the discrete case established  in \cite{KocSid}. We consider the beginning of the $5$-term exact sequence associated to the LHS spectral sequence arising from the central extension of pro-$p$ groups $W_p \rightarrowtail L_p \twoheadrightarrow \rho_p(L_p)$, i.e. there is an exact sequence
  $$H_2(\rho_p(L_p), \mb{Z}_p) \to H_0(\rho_p(L_p), H_1(W_p, \mb{Z}_p)) \to H_1(L_p, \mb{Z}_p) \to H_1(\rho_p(L_p), \mb{Z}_p) \to 0$$
  and since $W_p$ is central $L_p$ we have $H_0(\rho_p(L_p), H_1(W_p, \mb{Z}_p)) \simeq H_1(W_p, \mb{Z}_p) \simeq W_p$.
  It follows that $W_p$ is finitely generated as a pro-$p$ group if 
  $H_2(\rho_p(L_p),\mb{Z}_p)$ and $H_1(L_p, \mb{Z}_p) \simeq L_p/\overline{[L_p,L_p]}$ are finitely generated. 
  Note that by Proposition \ref{Lfg} $L_p$ is a finitely generated pro-$p$ group.
  
  It remains to prove that  $H_2(\rho_p(L_p),\mb{Z}_p)$ is a finitely generated pro-$p$ group. Note that
  $$
  \rho_p(L_p) \simeq \overline{\langle \{ (g^{-1}, 1, g) \mid g \in G \} \rangle} \simeq \{ (g_1, g_2) \in G \times G \mid g_1 g_2 \in \overline{[G,G]}\}.$$
  This is exactly the group $S$ considered in Corollary \ref{homology}. So $\rho_p(L_p)$ is finitely presented and thus 
   $H_2(\rho_p(L_p),\mb{Z}_p)$ is finitely generated as we wanted.
%
 \end{proof}
\end{section}

\begin{section}{The Schur multiplier}

\subsection{The non-abelian exterior square}

In  \cite{Mor} Moravec defined the \textit{non-abelian tensor square} $G \widehat{\otimes} G$ of a pro-$p$ group $G$ as the pro-$p$ group generated (topologically) by the symbols $g \ten h$, for $g,h \in G$, subject to the defining relations
\begin{equation} \label{tensor1} (g_1 g) \ten h= (g_1^g \ten h^g) (g \ten h) \end{equation} 
and
\begin{equation}  \label{tensor2} g \ten (h_1 h) = (g \ten h)(g^h \ten h_1^h) \end{equation}
for all $g, g_1,h, h_1 \in G$. Furthermore the \textit{non-abelian exterior square} $G \widehat{\wedge} G$ of $G$ is
\[ G \widehat{\wedge} G = G \widehat{\otimes} G / \Delta(G),\]
where $\Delta(G)$ is the normal (closed) subgroup generated by $g \ten g$, for all $g \in G$. We denote by $g \ext h$ the image of $g \ten h$ in $G \widehat{\wedge} G$.

Let $$\mu_G: G  \widehat{\wedge}  G \to \overline{[G,G]}$$ be the homomorphism of pro-$p$ groups  defined by $$\mu_G( g \ext h) = [g,h]\hbox{ for all }g,h \in G.$$ 

\begin{prop}[\cite{Mil}; \cite{Mor}, Proposition 2.2]
 For any pro-$p$ group $G$ we have
 \[ H_2(G, \mb{Z}_p) \simeq Ker(\mu_G).\]
\end{prop}

\subsection{The subgroup \texorpdfstring{$R_p(G)$}{RpG}}
Consider the pro-$p$ group 
\[R_p=R_p(G) = \overline{[G,[L_p(G), G^{\psi}]]} \subseteq \X_p(G).\]
As in the discrete case, it is immediate that $R_p(G) \subseteq W_p(G)$, thus it is an abelian subgroup of $\X_p(G)$. Also, we can show that $R_p$ is actually normal in $\X_p(G)$, as a consequence of the fact that $D_p$ and $L_p$ commute.

\begin{lemma} \label{calculation} In $\X_p(G)$ for $x,y, g,h,k \in G$ we have
	$[x, y^{\psi}] = [x^{\psi},y]$ and $[ [g,h^{\psi}], k]  \in [ [g,h], k^{\psi}] R_p$.
\end{lemma} 

\begin{proof} The formula 	$[x, y^{\psi}] = [x^{\psi},y]$ is proved for $\X(G)$ in \cite[Lemma~4.1.6]{Sid} using the commutator formulas  \begin{equation} \label{commutator001} [ab,c] = [a,c]^b. [b,c] \hbox{ and }[a,bc] = [a,c]. [a,b]^c \end{equation} The same proof applies here.
	 	
	Using $[x,y^{\psi}] = [x^{\psi},y]$ we obtain
		 $$[ [g,h], k^{\psi}] = [ [g^{\psi},h^{\psi}], k] = [[ g (g^{-1} g^{\psi}), h^{\psi}],k] = [ [g,h^{\psi}]^{g^{-1} g^{\psi}} . [g^{-1} g^{\psi}, h^{\psi}],k] = : \alpha.$$ Using $[[g,h^{\psi}], g^{-1} g^{\psi}] \in [D_p, L_p] = 1$ and $[[[g,h^{\psi}],k], [g^{-1} g^{\psi}, h^{\psi}]] \in [D_p, L_p] = 1$ we get $$\alpha =
		 [ [g,h^{\psi}]^{g^{-1} g^{\psi}} . [g^{-1} g^{\psi}, h^{\psi}],k] =
		    [ [g,h^{\psi}]  [g^{-1} g^{\psi}, h^{\psi}],k] = $$ $$ [[g,h^{\psi}],k]^{[g^{-1} g^{\psi}, h^{\psi}]}. [ [g^{-1} g^{\psi}, h^{\psi}],k] = [[g,h^{\psi}],k] [ [g^{-1} g^{\psi}, h^{\psi}],k] \in [[g,h^{\psi}],k] R_p. $$
\end{proof}

\begin{prop}
There is an epimorphism of pro-$p$ groups $H_2(G,\mb{Z}_p) \twoheadrightarrow W_p/R_p$.
\end{prop}

\begin{proof}
 Consider the homomorphism of pro-$p$ groups
 \[ \phi: G \widehat{\wedge} G \to \X_p(G)/ R_p\]
 defined by $\phi(g \ext h) = [g,h^{\psi}]$ for $g, h \in G$. The fact that $\phi$ is well-defined follows from (\ref{tensor1}), (\ref{tensor2}) and (\ref{commutator001}). Recall that $\alpha: \X_p(G) \to G$ is the homomorphism defined by $\alpha(x)=\alpha(x^{\psi}) = x$ and $L_p = Ker (\alpha)$. Now, consider the commutative diagram:
 \[
 \xymatrix{ G \widehat{\wedge} G \ar[r]^{\mu_G} \ar[d]_{\phi} & \overline{[G,G]} \ar[d]^{inc} \\
        \X_p(G)/ R_p \ar[r]_{\overline{\alpha}} & G}
 \]
 where $inc: \overline{[G,G]} \to G$ is the inclusion map and the homomorphism $\overline{\alpha}$ is induced by $\alpha$. It follows that $\phi(Ker(\mu_G))\subseteq Ker(\overline{\alpha}) = Ker(\alpha)/ R_p = L_p / R_p$. Since $\phi$ clearly takes values in $D_p/ R_p$, we have $\phi(Ker(\mu_G)) \subseteq (L_p \cap D_p) / R_p =  W_p / R_p$.
 Thus $\phi$ induces
 \begin{equation}  \label{hom1}
 \bar{\phi}: H_2(G,\mb{Z}_p) \simeq Ker(\mu_G) \to W_p/R_p.
 \end{equation}
 To see that $\bar{\phi}$ is surjective, we only need to observe that $Im(\phi)$ generates $D_p$ modulo $R_p$. Indeed,
 $D_p$ is generated \textit{as a normal closed subgroup} by the elements $[g,h^{\psi}]$, with $g, h \in G$, but modulo $R_p$ we have 
 \[ [g,h^{\psi}]^k \equiv [g,h^{\psi}] [ [g,h], k^{\psi}] \hbox{ for all }g,h,k \in G.\] 
 This follows from Lemma \ref{calculation} and $ [g,h^{\psi}]^k =  [g,h^{\psi}]  [[g,h^{\psi}],k]$.
 
 We conclude that the image of $\{[g,h^{\psi}]; g,h \in G\}$ generates (topologically) the quotient $D_p/R_p$. Thus for any $w \in W_p$, there is some $\xi \in G \widehat{\wedge} G$ such that $\phi(\xi)  = w R_p \subseteq W_p / R_p \subseteq L_p/ R_p = Ker (\overline{\alpha})$ and, by the commutativity of the diagram, $\xi \in Ker(\mu_G)$.
\end{proof}

\begin{theorem} \label{Schur}
	For any pro-$p$ group $G$
	 we have $H_2(G,\mb{Z}_p) \simeq W_p(G)/R_p(G)$.
\end{theorem}

\begin{proof}

In order to construct an inverse map $W_p/R_p \twoheadrightarrow H_2(G,\mb{Z}_p)$ we can proceed as in \cite[Subsection~6.2]{Men}, where the case of Lie algebras is considered. An outline of the procedure is as follows. We start by building a stem extension of pro-$p$ groups
\begin{equation}  \label{stem}
 1 \to H_2(G, \mb{Z}_p) =: M \hookrightarrow H \to  G \to 1,
\end{equation}
i.e. a central extension of pro-$p$ groups with $M \subseteq \overline{[H,H]}$.
This can be done by considering a free presentation $G = F/N$ (in the pro-$p$ sense), i.e. $F$ is a free pro-$p$ group, and taking $H = F/A$, where $A$ is 
a pro-$p$ subgroup of $N$ that contains $\overline{[F,N]}$ such that $A/\overline{[F,N]}$ is a complement of $H_2(G,\mb{Z}_p) \simeq  {(\overline{[F,F]} \cap N)}/{\overline{[F,N]}}$ inside $N/\overline{[F,N]}$. This complement exists because $N/(N \cap \overline{[F,F]}) \subseteq F/ \overline{[F,F]}$ is free pro-$p$ abelian and $N/\overline{[F,N]}$ is a pro-$p$ abelian group. Note that we have used that by  \cite[Thm.~4.3.4]{RibZal} every torsion-free pro-$p$ abelian group is free as a pro-$p$ abelian group, in particular this applies for $N/(N \cap \overline{[F,F]})$ that is torsion-free since $ F/ \overline{[F,F]}$ is torsion-free.

Now consider the map $$\rho_p^H: \X_p(H) \to H \times H \times H,$$ i.e. this is the map $\rho_p$ for $G$ substituted with $H$. By composing it with the projection onto the quotient $T= Im(\rho_p^H)/B$, where $B$ is defined by
 \[B = \{(x,xy,y) \hbox{ } | \hbox{ } x,y \in M\},\]
we obtain a map $\theta$ that factors through $\X_p(G)$:
\[ \xymatrix{ \X_p(H) \ar@{->>}[d]_{\pi} \ar@{->>}[r]^{\theta}  & T \\
\X_p(G) \ar@{->>}[ur]_{\lambda} & }
\]
where $\pi$ is induced by the epimorphism $\gamma : H \twoheadrightarrow G$ with kernel $M$. It is not hard then to verify $R_p(G) \subseteq Ker(\lambda)$ and 
\[ \lambda( W_p(G) ) = \{ (1,m,1) B \hbox{ } | \hbox{ } m \in M\} \simeq M.\]
Indeed 

1) $\rho_p^H(R_p(H)) \subseteq \rho_p^H(W_p(H)) = 1$, hence $\theta(R_p(H)) = 1$. And since $\pi(R_p(H)) = R_p(G)$ we conclude that $\lambda (R_p(G)) = \theta(R_p(H)) = 1$.

2) Let $\delta : T \to Im (\rho_p) \subseteq G \times G \times G$ be the map induced by $\gamma$. Then $\delta \lambda = \rho_p = \rho_p^G$ has kernel $W_p(G)$, hence $\lambda(W_p(G)) = Ker (\delta) = Ker (\nu)/ B$, where $\nu : Im (\rho_p^H) \to Im (\rho_p^G)$ is the epimorphism induced by $\gamma$. Note that $Ker(\nu) = (M \times M \times M) \cap Im (\rho_p^H) = (M \times M \times M) \cap \{ (h_1, h_2, h_3) \mid h_1 h_2^{-1} h_3 \in \overline{[H,H]} \} =  M \times M \times M$, hence $\lambda( W_p(G) ) = (M \times M \times M) / B \simeq M$ as required.

Thus $\lambda$ induces a map $$\bar{\lambda}: W_p(G)/R_p(G) \to M \simeq H_2(G,\mb{Z}_p).$$ If we realize both the homomorphism 
\eqref{hom1} and the stem extension \eqref{stem} by means of the Hopf formula for a fixed free presentation $G = F/N$, it is not hard to see that we have actually constructed maps that are inverse to each other. 
\end{proof}

Note that Theorem \ref{Schur} is consistent with the discrete case: if $G$ is a finite $p$-group, then the isomorphism $\X_p(G) \to \X(G)$ identifies $W_p(G)$ with $W(G)$ and $R_p(G)$ with $R(G)$, but also $H_2(G,\mb{Z}_p) \simeq H_2(G, \mb{Z})$ (the first homology is continuous, the second discrete).

\begin{section}{Pro-p analytic groups}
	
	For a pro-$p$ group $G$, we denote by $d(G)$ the cardinality of a minimal (topological) generating set. Furthermore, 
	$rk(G)$ denotes its \textit{rank}, that is
	\[ rk(G) = sup \{ d(H)| H \leqslant_c G\}.\]
	By \cite{LubMan}, a pro-$p$ group $G$ is $p$-adic analytic if and only if $rk(G) <\infty$. From now on we write analytic for $p$-adic analytic.
	\begin{prop} \label{anal}
		Let $G$ be a pro-$p$ group. Then $\X_p(G)$ is analytic (resp. poly-procyclic) if and only if $G$ is analytic (resp. poly-procyclic).
	\end{prop}
	
	\begin{proof}
		Recall that the property of being analytic behaves well with respect to extensions, that is, if $N \rightarrowtail G \twoheadrightarrow Q$ is an
		exact sequence of pro-$p$ groups, then $G$ is analytic if and only if both $N$ and $Q$ are (\cite[Corollary~2.4]{LubMan}). Thus one implication of the proposition is immediate since $G$ is a quotient of $\X_p(G)$.
		
		Suppose that $G$ is analytic. Then  $G$ is of type $FP_{\infty}$ (see \cite{Kin} or \cite{S-W} for instance) and $\overline{[G,G]}$ is finitely generated as a pro-$p$ group. Thus Proposition \ref{Wfg} applies and $W_p(G)$ is a finitely generated abelian pro-$p$ group. In particular, $W_p(G)$ is also analytic. But $\X_p(G)/W_p(G) \simeq Im(\rho_p)$ must be analytic as well, being a closed subgroup 
		of $G \times G \times G$. Thus $\X_p(G)$ is analytic.
		
		The result for poly-(procyclic) pro-$p$ groups follows from the observation that these groups are exactly the soluble pro-$p$ groups of finite rank (\cite[Proposition~8.2.2]{Wil}). Thus we only need to combine the first part of the proof with Proposition \ref{Solv}. 
	\end{proof}

	\begin{crlr} \label{cook-}
		Suppose that $G$ is soluble pro-$p$ group of type $FP_{\infty}$. 
		Suppose further that $G$ is torsion-free or metabelian. Then $\X_p(G)$ is a soluble pro-$p$ group of type $FP_{\infty}$.
	\end{crlr}

\begin{proof} Note that $\X_p(G)$ is soluble since $G$ is soluble. 
	
	Assume first  that $G$ is torsion-free. By the main result of Corob Cook in \cite{Cor} $G$ is analytic, hence by Proposition \ref{anal} $\X_p(G)$ is analytic and so is of type $FP_{\infty}$. 
	
	Suppose that $G$ is metabelian. Then it fits into an exact sequence of pro-$p$ groups
	\[ 1 \to A \to G \to Q \to 1,\]
	where $A$ is abelian and $Q$ has finite rank (and in our case is abelian). By a result of King (\cite[Theorem~6.2]{Kin}), we know that if $G$ is of type $FP_{\infty}$, then $G$ is actually itself of finite rank, so is analytic. Then we can apply again Proposition \ref{anal} to deduce that $\X_p(G)$ is analytic and so is of type $FP_{\infty}$. 
	\end{proof}

It is plausible that Corob Cook result from \cite{Cor} holds for pro-$p$ groups that are not torsion-free but this is still an open problem. If that is the case then the condition in Corollary \ref{cook-} that $G$ is torsion-free is redundant.
\end{section}

\section{The group $\nu_p(G)$} \label{sectionRocco} 

In  \cite{Norai} Rocco defined for an arbitrary discrete group $H$ the discrete group $\nu(H)$. In \cite{Ellis} Ellis and  Leonard studied a similar construction.

For   a pro-$p$ group $G$ we define the pro-$p$ group $\nu_p(G)$ by the pro-$p$ presentation
$$\nu_p(G) = \langle G, G^{\psi} \mid [g_1, g_2^{\psi}]^{g_3} = [ g_1^{g_3}, (g_2^{g_3})^{\psi}] = [g_1, g_2^{\psi}]^{g_3^{\psi}}\rangle_p, $$
where $G^{\psi}$ is an isomorphic copy of $G$. Set $\overline{[G, G^{\psi}]}$ as the pro-$p$ subgroup of $\nu_p(G)$ generated by $\{ [g_1, g_2^{\psi}] | g_1, g_2 \in G  \}$.
The proof of Lemma \ref{inverse0} can be adapted to prove the following lemma.

\begin{lemma} \label{inverse} Let $G$ be a pro-$p$ group. Then $\nu_p(G)$ is the inverse limit of $\nu(G/ U)$, where $G$ is the inverse limit of finite $p$-groups $G/ U$. In particular, the pro-$p$ subgroup $\overline{[G, G^{\psi}]}$ of $\nu_p(G)$ is the inverse limit of the finite $p$-subgroups  $[G/U, (G/U)^{\psi}]$ of $\nu(G/U)$,  where $G$ is the inverse limit of finite $p$-groups $G/ U$.\end{lemma} 

We will need in this section the following homological result.

\begin{lemma} \label{homological} Let $ A \to B \to C$ be a stem extension of pro-$p$ groups i.e. a short exact sequence of pro-$p$ groups such that $A \subseteq \overline{[B,B]} \cap Z(B)$. Then $A$ is a pro-$p$ quotient of $H_2(C, \mathbb{Z}_p)$.
\end{lemma}

\begin{proof}
	We consider the beginning of the $5$-term exact sequence associated to the LHS spectral sequence arising from $ A \to B \to C$, i.e. there is an exact sequence
	$$H_2(C, \mb{Z}_p) \to H_0(C, H_1(A, \mb{Z}_p)) \to H_1(B, \mb{Z}_p) \to H_1(C, \mb{Z}_p) \to 0.$$
	Since $A \subseteq \overline{[B,B]}$ and for a pro-$p$ group $G$ we have that $H_1(G, \mathbb{Z}_p) \simeq G/ \overline{[G,G]}$ we deduce that $H_1(B, \mb{Z}_p) \to H_1(C, \mb{Z}_p)$ is an isomorphism of pro-$p$ groups. Thus the map $H_2(C, \mb{Z}_p) \to H_0(C, H_1(A, \mb{Z}_p))$ is an epimorphism. Since $A \subseteq  Z(B)$ we deduce that the $C$ action (via conjugation) on $A$ is trivial and $A$ is abelian, hence $H_0(C, H_1(A, \mb{Z}_p)) \simeq H_1(A, \mb{Z}_p) \simeq A$.	
\end{proof}

Let $\Delta_p(G)$ be the pro-$p$ subgroup of $\nu_p(G)$ (topologically) generated by $\{ [g, g^{\psi}] \mid g \in G \}$. 
The same commutator calculations from \cite[Lemma~2.1]{Norai} show that 
	\begin{equation} \label{central} \Delta_p(G) \subseteq \nu_p(G)' \cap Z(\nu_p(G))\end{equation}
and the argument from \cite[p.~69]{Norai} gives that there is an isomorphism of pro-$p$ groups \begin{equation} \label{iso99}  \X_p(G)/ R_p(G) \simeq \nu_p(G)/ \Delta_p(G)  \end{equation} induced by the map that is identity on $G \cup G^{\psi}$. By (\ref{central}), (\ref{iso99}) and Lemma \ref{homological} we have that $\Delta_p(G)$ is a quotient of $H_2(\nu_p(G)/ \Delta_p(G), \mathbb{Z}_p)$, thus
\begin{equation} \label{quotient}  \Delta_p(G) \hbox{ is a quotient of } H_2(\X_p(G)/ R_p(G), \mathbb{Z}_p). \end{equation} The short exact sequence of pro-$p$ groups
$ 1 \to W_p(G)/ R_p(G) \to \X_p(G)/ R_p(G) \to \X_p(G)/ W_p(G) \to 1 $ can be written as
\begin{equation}  \label{eq-neu} 1 \to H_2(G, \mathbb{Z}_p) \to \X_p(G)/ R_p(G) \to Im (\rho_p) \to 1  \end{equation}

\begin{lemma}
	Let $G$ be a pro-$p$ group. 
	 Then the map $\varphi : G \widehat{\otimes} G \to \overline{[G, G^{\psi}]}$ given by $g_1 \widehat{\otimes} g_2 \to [g_1, g_2^{\psi}]$ is an isomorphism between the non-abelian pro-$p$ tensor square  $ G \widehat{\otimes} G$ and the 	pro-$p$ subgroup $\overline{[G, G^{\psi}]}$ of $\nu_p(G)$.	
\end{lemma}

\begin{proof} The fact that $\varphi$ is well-defined follows from the defining relations of $G \widehat{\otimes} G$  and $\nu_p(G)$.
	Since $G \widehat{\otimes} G$ is the inverse limit of $G/U {\otimes} G/U$ \cite{Mor} and $\overline{[G, G^{\psi}]}$ is the inverse limit of $[G/U, (G/U)^{\psi}]$, where both inverse limits are over the set of all finite quotients $G/ U$ of $G$, it suffices to use the fact that for finite $p$-groups $P = G/ U$ there is an isomorphism $P {\otimes} P \to [P, P^{\psi}]$ given by $p_1 {\otimes} p_2 \to [p_1, p_2
	]$ by \cite[Prop.~2.6]{Norai}.
\end{proof}

{\bf Proof of Theorem C}

	By Lemma \ref{finite-p}, Corollary \ref{fin-pres}, Proposition \ref{nilpotent1}, Proposition \ref{Solv} and Proposition \ref{anal}  $\X_p(G) \in {\mathcal{P}}$. 
	
	1. Suppose $\mathcal P$ is not the class of finitely presented pro-$p$ groups. Then $\nu_p(G)/ \Delta_p(G) \simeq \X_p(G)/ R_p(G) \in {\mathcal P}$. 
	
	1.1 If $\mathcal P$ is the class of finitely generated nilpotent groups or the class of soluble groups and since $\Delta_p(G)$ is central in $\nu_p(G)$, we get that $\nu_p(G) \in {\mathcal P}$.
	
	1.2. If ${\mathcal P}$ is the class of poly-procyclic pro-$p$ groups or more generally the class of analytic pro-$p$ groups we have $H_2(\X_p(G)/ R_p(G), \mathbb{Z}_p) \in \mathcal{P}$ and by (\ref{quotient}) we get $\Delta_p(G) \in \mathcal{P}$. Thus $\nu_p(G) \in {\mathcal P}$.
	
	Finally in both cases 1.1 and 1.2 ${\mathcal P}$ is subquotient closed, thus $ G \widehat{\otimes} G \simeq \overline{[G, G^{\psi}]}  \in {\mathcal P}$.
	
	2. Suppose that ${\mathcal P}$ is the class of finitely presented pro-$p$ groups. 
	Then $G$ is a finitely presented pro-$p$ group, so $H_2(G, \mathbb{Z}_p)$ is finitely generated abelian  pro-$p$ group. Furthermore by  Corollary A applied for $k = 2$ and the fact that for a pro-$p$ group $FP_2$ and finite presentability are the same, $Im (\rho_p)$ is a finitely presented pro-$p$ group.  Thus by (\ref{eq-neu}) $\X_p(G)/ R_p(G)$ is finitely presented,
	 hence $H_2(\X_p(G)/ R_p(G), \mathbb{Z}_p)$ is a finitely generated abelian pro-$p$ group and by (\ref{quotient}) $\Delta_p(G)$ is a finitely generated abelian group. By (\ref{iso99}) $\nu_p(G)/ \Delta_p(G)$ is a finitely presented pro-$p$ group and this implies that $\nu_p(G)$ is finitely presented too. This completes the proof of Theorem C.
	 
	 \medskip
 In \cite{Mor2} Moravec proved that the non-abelian tensor product of powerful $p$-groups acting powerfully and compatibly upon each other is again a powerful $p$-group. In particular if $G$ is a powerful pro-$p$ group, every finite quotient $G/ U$ is a powerful $p$-group, hence $G/U \otimes G/U$ is a powerful $p$-group. Since $G \widehat{\otimes} G$ is the inverse limit of $G/U \otimes G/U$, where $U$ runs through all finite index normal subgroups, we deduce that $G \widehat{\otimes} G$ is a powerful pro-$p$ group.
Still $\nu_p(G)$ does not need to be powerful. For example,   consider $G = \mathbb{Z}_p/ p \mathbb{Z}_p$ for $p > 2$. Then $\nu_p(G) = \nu(G)$ is a nilpotent of class 2 $p$-group with  commutator $[G, G^{\psi}] \simeq \mathbb{Z}_p/ p \mathbb{Z}_p$, $\nu_p(G)$ has exponent $p$ and is not abelian. Thus $\nu_p(G)$ is not powerful.

\end{section}

\begin{section}{\texorpdfstring{$\X_p$}{Xp} does not preserve \texorpdfstring{$FP_3$}{FP3}}
We begin by proving an auxiliary result about $R_p(G)$, which may be of independent interest.

\begin{prop}  \label{Rtrivial}
If $G$ is a $2$-generated pro-$p$ group, then $R_p(G) = 1$.
\end{prop}
 
This can be obtained as a corollary of the analogous result for discrete groups, which is a consequence of Lemma \ref{Z1}. The discrete case of Proposition \ref{Rtrivial} was first proved with different methods by Bridson and Kochloukova in 
\cite{BriKoc2}. We write $y^{- \psi}$ for $(y^{\psi})^{-1}$.

\begin{lemma} \label{Z1}
 Let $G$ be a discrete group and let $X \subset G$ be a generating set. Suppose that $X$ is symmetric (with respect to inversion). Then $R(G) = [G, [L, G^{\psi}]]$ is the normal subgroup of $\X(G)$ generated by the set
 \[ Z = \{ [x_1, [yy^{-\psi}, x_2^{\psi}]] \ \ | \ \ x_1, x_2, y \in X\}.\] 
\end{lemma}

\begin{proof} The proof relies on commutator calculations that use the following commutator identities
	$$[a,bc] = [a,c]. [a,b]^c \hbox{ and } [ab,c] = [a,c]^b. [b,c]$$
 Recall that $R(G)$ is generated as a subgroup by the elements of the form $r=[g,[\ell,h^{\psi}]]$, with $g,h\in G$ and 
 $\ell \in L = L(G)$. If $g=g_1 g_2$, then $r$ is a consequence of $[g_1,[\ell,h^{\psi}]]$ and $[g_2,[\ell,h^{\psi}]]$. The analogous claim holds ``in the other variable'', that is, for $h=h_1 h_2$. 
 Indeed $[g,[\ell,(h_1 h_2)^{\psi}]]$ is a consequence of $[g,[\ell,h_2^{\psi}]]$ and $$r_1 = [g,[\ell,h_1^{\psi}]^{h_2^{\psi}}] =  [g [g, h_2^{- \psi}],[\ell,h_1^{\psi}]]^{h_2^{\psi}} = [g,[\ell,h_1^{\psi}]]^{ [g, h_2^{- \psi}] h_2^{\psi}},$$
 where the last equality follows from the fact that $[[g, h_2^{- \psi}],[\ell,h_1^{\psi}]] \in [D,L] = 1$, where $D = D(G)$.
 
 If $ \ell = \ell_1 \ell_2$, with 
 $\ell_2=yy^{-\psi}$, then by applying the commutator formulas we obtain that $r$ is a consequence of $[g,[yy^{-\psi},h^{\psi}]]$ and $r_2=[g,[\ell_1,h^{\psi}]^{yy^{-\psi}}]$. But
\[r_2 = [g, ([\ell_1,h^{\psi}] [ [\ell_1,h^{\psi}],y ] )^{y^{-\psi}}],\]
 thus $r_2$ is a consequence of $[y,[\ell_1,h^{\psi}]]$ and $r_3=[g,[\ell_1,h^{\psi}]^{y^{-\psi}}]$. Again
\[ r_3 = [g[g,y^{\psi}],[\ell_1,h^{\psi}]]^{y^{-\psi}}  = [g,[\ell_1,h^{\psi}]]^{[g,y^{\psi}] y^{-\psi}},\]
where the last equality follows from the fact that $[[g,y^{\psi}],[\ell_1,h^{\psi}]] \in [D,L] = 1$.
A similar argument works for $ \ell = \ell_1 \ell_2$, with 
$\ell_2=(yy^{-\psi})^{-1}$.

Finally, if $r=[g, [bb^{-\psi},h^{\psi}]]$ for some $b=uv \in G$, then 
\[ r = [g, [(vv^{-\psi} u^{-\psi}u)^{u^{-1}}, h^{\psi}]] = [g', [vv^{-\psi} uu^{-\psi}, (h^{\psi})^u]]^{u^{-1}}\] 
for $g' = g^u$. But 
\[ [vv^{-\psi} uu^{-\psi}, (h^{\psi})^u] = [vv^{-\psi} uu^{-\psi}, [u,h^{-\psi}] h^{\psi}] = [vv^{-\psi} uu^{-\psi}, h^{\psi}],\]
since $ [vv^{-\psi} uu^{-\psi}, [u,h^{-\psi}]] \in [L,D]=1$. Thus $r$ is a consequence of $[g', [vv^{-\psi} uu^{-\psi}, h^{\psi}]]$, and we fall in the previous case, that is, $r$ is a consequence of $[g', [vv^{-\psi}, h^{\psi}]]$ and $[g', [uu^{-\psi}, h^{\psi}]]$.

The arguments above imply that any $r=[g,[\ell, h^{\psi}]]$ is a consequence of the elements of $Z$ i.e. is in the normal subgroup generated by $Z$.
\end{proof}

\begin{lemma} \label{Z2} If $G$ is a discrete group generated by $\{x,y\}$, then the generators of $R(G) $ as a normal subgroup in   $\X(G)$, given by Lemma \ref{Z1}, are all trivial i.e. $Z = 1$. \end{lemma}
\begin{proof}  Note that we can take $X = \{x, y,x^{-1}, y^{-1}\}$. Consider the generator 
\[ r = [x, [yy^{-\psi},x^{\psi}]].\]
 By the standard commutator identities we have
\begin{equation} \label{ball1}   r = [x,[y^{-\psi},x^{\psi}]] [x, [y,x^{\psi}]^{y^{-\psi}}]^{[y^{-\psi},x^{\psi}]}. \end{equation} 
Using the identity $[x,y^{\psi}]=[x^{\psi},y]$ we deduce that $[x, [y,x^{\psi}]^{y^{-\psi}}] = [x, [y^{\psi},x]^{y^{-\psi}}]$. Furthermore since $x^{-1} x^{\psi} \in L$ we have $[y^{-\psi},x^{\psi}] \in [y^{-\psi},x ] L$. Then since $[D,L]=1$ and $[x, [y^{\psi},x]^{y^{-\psi}}] \in D$, we can reduce the 
second term of the above product in  (\ref{ball1}) to $[x, [y^{\psi},x]^{y^{-\psi}}]^{[y^{-\psi},x]}.$ 

Applying the Hall-Witt identity, we can see that the first term of the above product in  (\ref{ball1}) reduces to $[x,[y^{-\psi},x]]$.
Indeed
\[ 1 = [[y^{-\psi},x^{\psi}],x]^{x^{-\psi}} [[x^{-\psi}, x^{-1}], y^{-\psi}]^x [[x,y^{\psi}],x^{-\psi}]^{y^{-\psi}},\]
so 
\[ [x,[y^{-\psi},x^{\psi}]] = ([[y^{-\psi},x^{\psi}],x])^{-1} = [[x,y^{\psi}],x^{-\psi}]^{y^{-\psi}x^{\psi}}.\]
%
Now using twice that $[D,L]=1$ we have 
\[ [x,[y^{-\psi},x^{\psi}]] = [[x,y^{\psi}],x^{-\psi}]^{y^{-\psi}x^{\psi}} = [[x,y^{\psi}],x^{-1}]^{y^{-\psi}x^{\psi}} = [[x,y^{\psi}],x^{-1}]^{y^{-\psi}x}.\]
Similarly, the Hall-Witt identity gives
\[[x,[y^{-\psi},x]] = ([[y^{-\psi},x],x])^{-1} = [[x,y^{\psi}],x^{-1}]^{y^{-\psi}x},\]
so $[x,[y^{- \psi}, x^{\psi}]] = [x,[y^{- \psi}, x]]$ and by (\ref{ball1}) we get 
\begin{equation} \label{ball2} r = [x, [y^{- \psi},x]].[x, [y^{\psi},x]^{y^{- \psi}} ]^{[y^{- \psi},x]} \end{equation}
Using the commutator formulas and (\ref{ball2}) we obtain
$$1 = [x, [y^{\psi} y^{-\psi},x]] = [x, [y^{\psi},x]^{y^{- \psi}} . [y^{- \psi},x]] =[x, [y^{- \psi},x]].[x, [y^{\psi},x]^{y^{- \psi}} ]^{[y^{- \psi},x]} = r$$
All the other generators of $R(G)$, given by Lemma \ref{Z1}, are either immediately trivial or can be shown to be trivial exactly as above. So $R(G)=1$. \end{proof}

\noindent
{\it Proof of Proposition \ref{Rtrivial}} Let $\widehat{F}$ be the free pro-$p$ group on two generators and $F$ be the free discrete group on two generators. Then we can write $\X_p(\widehat{F})$ as an inverse limit of the quotients $\X(P)$, where each $P$ is a $2$-generated $p$-group and by Lemma \ref{Z2} $R(P) = 1$. Clearly $R(P)$ is the image of $R_p(\widehat{F})$ by the  homomorphism $\X_p(\widehat{F}) \twoheadrightarrow \X(P)$ induced by the epimorphism $\widehat{F} \to P$. Then it follows that $R_p(\widehat{F})$ must be trivial as well. In general, if $G$ is a $2$-generated pro-$p$ group, then $R_p(G)$ is a quotient of $R_p(\widehat{F})$,
so $R_p(G)=1$. This completes the proof of Proposition \ref{Rtrivial}.

\begin{prop} \label{fp333}
 If $G$ is a non-abelian free pro-$p$ group, then $\X_p(G)$ is not of type $FP_3$.
\end{prop}

\begin{proof}
Let $\widehat{F}$ be the free pro-$p$ group of rank $2$. Then $R_p(\widehat{F})=1$, which combined with Theorem \ref{Schur} and  $H_2(\widehat{F}, \mathbb{Z}_p) = 0$ implies that $W_p(\widehat{F}) = 1$ and hence $\X_p( \widehat{F} ) \simeq Im(\rho_p^{\widehat{F}})$. Note that $Im(\rho_p^{\widehat{F}})$ is a subdirect product of 
$\widehat{F} \times \widehat{F} \times \widehat{F}$ which is clearly not of type $FP_3$ (by Theorem A in \cite{KocSho}, for instance) but is $FP_2$ by Corollary A. Then $H_i(\X_p(\widehat{F}), \mathbb{F}_p)$ is finite dimensional for $i \leq 2$ and is infinite dimensional for $i = 3$.

 More generally, if $G$ is any non-abelian free pro-$p$ group, then the epimorphism
$\pi: \X_p(G) \to \X_p(\widehat{F})$ induced by any epimorphism $\gamma: G \twoheadrightarrow \widehat{F}$ is split, i.e.  there is a homomorphism $\sigma : \X_p(\widehat{F}) \to \X_p(G)$ such that  $\pi \circ \sigma = id$ and $\sigma$ is induced by a splitting of $\gamma$. The same holds for the induced homomorphisms on the homologies. In particular $\pi_{\ast}: H_3( \X_p(G); \mb{F}_p) \to H_3(\X_p(\widehat{F}); \mb{F}_p)$ is surjective, thus $H_3( \X_p(G), \mb{F}_p)$ is not finite-dimensional if $H_3(\X_p(\widehat{F}),\mb{F}_p)$ is not. Thus $\X_p(G)$ cannot be of type $FP_3$ either.
\end{proof}

{\bf Proof of Theorem B} It follows by Lemma \ref{finite-p}, Corollary \ref{fin-pres}, Proposition \ref{nilpotent1}, Proposition \ref{Solv}, Proposition \ref{anal} and Proposition \ref{fp333}.
\end{section}

\end{document}